\documentclass[a4paper,12pt,oneside]{amsart}       

\usepackage[british,english]{babel} 
\usepackage{amsmath,amssymb,amsthm,amscd}   
\usepackage{mathrsfs}
\usepackage[all]{xy}
\usepackage{stmaryrd}

\theoremstyle{cited}
\newtheorem{teor}{Theorem}[section]
\newtheorem{lem}[teor]{Lemma}
\newtheorem{cor}[teor]{Corollary}

\theoremstyle{definition}
\newtheorem{deft}[teor]{Definition}

\theoremstyle{remark}
\newtheorem{oss}[teor]{Remark}

\newcommand{\bbR}{\mathbb{R}}

\newcommand{\bbN}{\mathbb{N}}
\newcommand{\bound}{\partial_\infty}

\newcommand{\hyp}{\mathbb{H}}
\newcommand{\bir}{\textup{cr}}
\newcommand{\cat}{\textup{CAT}}
\newcommand{\isom}{\textup{Isom}}
\newcommand{\veps}{\varepsilon}

\title[Asymptotically moebius maps]{Asymptotically Moebius maps and rigidity for the hyperbolic plane}

\author{Alessio Savini}

\begin{document}

\maketitle

\begin{abstract}
Let $S$ be a rank-one symmetric space of non-compact type and let $X$ be a $\cat (-1)$ space. A well-known result by Bourdon states that if a topological embedding $\varphi: \bound S \rightarrow \bound X$ respects cross ratios, that means $$\bir_S( \xi_0,\eta_0,\xi_1,\eta_1)=\bir_X( \varphi(\xi_0),\varphi(\eta_0),\varphi(\xi_1),\varphi(\eta_1))$$ for every $\xi_0,\eta_0,\xi_1,\eta_1 \in \bound S$, then $\varphi$ is induced by an isometric embedding of $S$ into $X$. 

We generalize this result when $S=\hyp^2$ is the real hyperbolic plane as it follows. Let $\varphi_k: \bound \hyp^2 \rightarrow \bound X$ be a sequence of continuous maps which are asymptotically Moebius, that means $$\lim_{k \to \infty} \bir_X(\varphi_k(\xi_0),\varphi_k(\eta_0),\varphi_k(\xi_1),\varphi_k(\eta_1))=\bir_{\hyp^2}( \xi_0,\eta_0,\xi_1,\eta_1)$$ for every $\xi_0,\eta_0,\xi_1,\eta_1 \in \bound \hyp^2$. Assume that the isometry group $\isom(X)$ acts transitively on triples of distinct points of $\bound X$. Then there must exists a sequence $(g_k)_{k \in \bbN}$, $g_k \in \isom(X)$ and a map $\varphi_\infty: \bound \hyp^2 \rightarrow \bound X$ such that $\lim_{k \to \infty} g_k\varphi_k(\xi)=\varphi_\infty(\xi)$ for every $\xi \in \bound \hyp^2$ and $\varphi_\infty$ is induced by an isometric embedding of $\hyp^2$ into $X$. 
\end{abstract}


\section{Introduction}

Given a $\cat (-1)$ space $X$ it is possible to define a real-valued function $\bir_X:(\bound X)^4 \rightarrow \bbR$ called cross ratio which takes quadruples of distinct points on the boundary at infinity of $X$. One of the possible ways of defining it relies on the existence on $\bound X$ of a family of visual metrics $\rho_x$, where $x \in X$, such that they all induce the same cross ratio (we say that they are Moebius equivalent). It is worth noticing that when $X$ is the universal cover of a surface, there are infinitely many inequivalent cross ratios on its boundary at infinity and they parametrize some classical geometric objects such as points of Teichmuller space (see~\cite{bonahon:articolo}) or Hitchin representations (see~\cite{labourie:articolo}). 

One of the main interest in the study of properties of the boundary at infinity $\bound X$ of a negatively curved space is the attempt to gain information about the geometry of $X$. A remarkable example of this philosophy is given by the rigidity result of~\cite{bourdon96:articolo} which states the following. Let $S$ be a rank-one symmetric space of non-compact type and let $X$ be a $\cat (-1)$ space.  Any Moebius embedding $\varphi: \bound S \rightarrow \bound X$, that means a topological embedding which satisfies $\bir_S( \xi_0,\eta_0,\xi_1,\eta_1)=\bir_X( \varphi(\xi_0),\varphi(\eta_0),\varphi(\xi_1),\varphi(\eta_1))$ for every $\xi_0, \eta_0,\xi_1,\eta_1 \in \bound S$, is induced by an isometric embedding $F:S \rightarrow X$.

The importance of Moebius embeddings relies on their deep relation with the marked length spectrum rigidity conjecture stated by Burns and Katok. In~\cite{burns:articolo} they conjectured that any isomorphism $\Phi: \pi_1(M) \rightarrow \pi_1(N)$ between the fundamental groups of two closed negatively curved manifolds $M,N$ which respects the associated length functions $\ell_M, \ell_N$, that is $\ell_M = \ell_N \circ \Phi$, must be induced by an isometry $F:M \rightarrow N$. One could interpret this conjecture as a variation of the Mostow rigidity theorem stated in~\cite{mostow:articolo} suitably adapted to manifolds with variable negative curvature. 

A proof of this conjecture in the case of surfaces was found independently by Otal in~\cite{otal90:articolo} and by Croke in~\cite{croke:articolo}, whereas the higher dimensional case is still open. Later both Otal and Hamenstadt found two remarkable equivalent formulations of the marked length spectrum rigidity. In~\cite{otal92:articolo} the former showed that the correspondence of marked length spectra for two negatively curved manifolds is equivalent to the existence of an equivariant Moebius map between the boundary at infinity of their universal covers. Similarly~\cite{ham:articolo} proved that the equivalence of marked length spectra implies a positive answer to the geodesic conjugacy problem which states that a homeomorphism $\Psi:T^1M \rightarrow T^1N$ between the unit tangent bundles of two negatively curved manifolds commuting with the geodesic flows is actually induced by an isometry. To sum up the marked length spectrum, the geodesic flow and the cross ratio on the boundary at infinity all encode the same geometric information about the negatively curved metric of a fixed manifold (see~\cite{biswas15:articolo} for a proof). 

For the reasons mentioned above many authors wanted to extend the result of Bourdon following several distinct directions. For instance both in~\cite{beyer:articolo} and in~\cite{fioravanti:articolo} the authors give a definition of cross ratio for flag boundaries of higher rank symmetric spaces and for the Roller boundary of $\cat (0)$ cube complexes. In addition they state an equivalent result to the one of Bourdon adapted to these specific contexts. Biswas studies in~\cite{biswas15:articolo,biswas:articolo} the extendability problem of Moebius maps between boundaries of proper, geodesically complete $\cat (-1)$ spaces. He shows that any Moebius homeomorphism can be extended to a $(1,\log 2)$-quasi isometry and this result can be refined for manifolds with pinched negative curvature.

In this paper we would like to focus our attention on a new perspective about Moebius maps. We set ourselves in the same hypothesis of Bourdon, that means let $S$ be a rank-one symmetric space of non-compact and let $X$ be a $\cat (-1)$ space. We say that a sequence of continuous maps $\varphi_k: \bound S \rightarrow \bound X$ is asymptotically Moebius if $\lim_{k \to \infty} \bir_X(\varphi_k(\xi_0), \varphi_k(\eta_0), \varphi_k(\xi_1), \varphi_k(\eta_1))=\bir_S(\xi_0,\eta_0,\xi_1,\eta_1)$ for every $\xi_0,\eta_1,\xi_1,\eta_1 \in \bound S$. It is pretty natural to ask under which hypothesis there should exists the limit of the sequence $(\varphi_k)_{k \in \bbN}$. Intuitively if this limit exists, it should be induced by an isometric embedding by the result of Bourdon. We point out that it may be necessary to precompose each map $\varphi_k$ with a suitable isometry $g_k \in \isom(X)$ in order to guarantee the existence of such a limit. 

We give an answer to the problem above when $S=\hyp^2$ is the hyperbolic plane and the isometry group $\isom(X)$ acts transitively on triples of distinct points in $\bound X$. More precisely we show the following 

\begin{teor}\label{convergence}
Let $X$ be a $\cat (-1)$ space and assume that $\isom (X)$ acts transitively on triples of distinct points on $\bound X$. Let $\varphi_k: \bound \hyp^2 \rightarrow \bound X$ be a sequence of continuous maps which is asymptotically Moebius. Then there must exists a sequence $(g_k)_{k \in \bbN}$ of isometries of $X$ and a map $\varphi_\infty: \bound \hyp^2 \rightarrow \bound X$ such that $\lim_{k \to \infty} g_k\varphi_k(\xi)=\varphi_\infty(\xi)$ for every $\xi \in \bound \hyp^2$. In particular $\varphi_\infty$ is induced by an isometric embedding. 
\end{teor}

The proof of this theorem is pretty easy and relies on the fact that the proof written by Bourdon can be $\veps$-perturbed remaining valid. In particular his approach allows us to construct by hand the limit map $\varphi_\infty$ we are looking for. Unfortunately the proof is specific for the hyperbolic plane and it cannot be directly extended to all the other symmetric spaces. 

The paper is structured as it follows. In the first section we recall the basic definitions we need, in particular the notion of visual metrics, the definition of cross ratio and its main properties. The second section is devoted to the proof of the main theorem and on some its consequences. We conclude with some comments about the possibility to extend the theorem to all rank-one symmetric spaces. 


\section{Preliminary definitions}

\subsection{Boundary of a $\cat (-1)$ space and cross ratio}
In this section we are going to recall the definitions of cross ratio and Moebius maps. For more details we refer mainly to~\cite{bourdon:tesi} and to~\cite{bridson:libro}.

Let be $(X,d_X)$ be a metric geodesic space. For every $x,y \in X$ we denote by $[x,y]$ the geodesic segment joining $x$ to $y$. It is possible to associate to every triangle $\Delta=[x,y] \cup [y,z] \cup [z,x]$ of $X$ a comparison triangle $\bar \Delta =[\bar x, \bar y] \cup [\bar y, \bar z] \cup [\bar z, \bar x]$ in the hyperbolic plane $\hyp^2$ whose sides have the same length of the sides of $\Delta$. This triangle is unique up to isometries. There exists a natural application $p: \Delta \rightarrow \bar \Delta$ whose restriction to each side of $\Delta$ is an isometry and that allows us to define the $\cat(-1)$ property. More precisely, we say that $\Delta$ satisfies the $\cat (-1)$ property if for every point $s,t \in \Delta$ we have that $$d_X(s,t) \leq d_{\hyp^2} (p(s),p(t)),$$ where $d_{\hyp^2}$ denotes the Riemannian distance on the hyperbolic plane. In the same way we say that $(X,d_X)$ is a $\cat (-1)$ space if each triangle $\Delta$ of $X$ satisfies the $\cat (-1)$ property. 

Every $\cat(-1)$ space $X$ admits a natural compactification. To construct the natural boundary at infinity we need first to define an equivalence relation between geodesic rays. Two geodesic rays $c_0,c_1:[0,\infty) \rightarrow X$ are \textit{asymptotic} if there exists a constant $C$ such that $d_X(c_0(t),c_1(t))<C$ for every $t>0$. The set $\bound X$ of boundary points of $X$, called \textit{boundary at infinity}, is the set of equivalence classes of asymptotic rays. If $c:[0,\infty) \rightarrow X$ is a geodesic ray, we sometimes denote its equivalence class by $c(\infty)$. 

Set $\bar X:= X \cup \bound X$. In order to endow $\bar X$ with a topology that realizes $\bar X$ as a compactification of $X$, we first fix an origin $x \in X$. Denote by $$\mathcal{R}_x=\{ c:I \rightarrow X| \textup{$c$ geodesic segment} \},$$ the set of all possible geodesic segments pointed at $x$, where the interval $I$ may be either equal to $I=[0,a], a \in \bbR$ or $I=[0,\infty)$. When $I=[0,a]$ for some $a \in \bbR$ we assume to extend the segment to the whole positive halfline by putting $c(t)=c(a)$ for every $t \geq a$. As before we put $c(\infty):=c(a)$. This choice allows us to define a natural map $$\textup{ev}_\infty: \mathcal{R}_x \rightarrow \bar X, \hspace{10pt} \textup{ev}_\infty(c):=c(\infty).$$ Since $\textup{ev}_\infty$ is surjective, if we endow $\mathcal{R}_x$ with the topology of uniform convergence on compact sets, we can give to $\bar X$ the quotient topology. With this topology, called \textit{cone topology}, $\bar X$ is compact and $X$ is an open dense subset. It is possible to show that the topology described so far does not depend on the choice of the origin $x \in X$.

Given any $x \in X$ and any $\xi \in \bound X$, there exists a geodesic ray pointed at $x$ representing $\xi$ . We are going to denote this ray by $[x,\xi)$. Similarly for two distinct points $\xi, \eta \in \bound X$ there exists a geodesic line joining them. We are going to denote this geodesic by $(\xi,\eta)$.

We are now ready to give the main definition of the section

\begin{deft}
Let $X$ be a $\cat (-1)$ space and let $\xi_0,\eta_0,\xi_1,\eta_1 \in \bound X$ be four distinct points on the boundary at infinty. Their \textit{cross ratio} is given by
\[
\bir_X(\xi_0,\eta_0,\xi_1,\eta_1)=\lim_{(x_0,y_0,x_1,y_1) \to (\xi_0,\eta_0,\xi_1,\eta_1)} e^{\frac{1}{2}(d_X(x_0,x_1)+d_X(y_0,y_1)-d_X(x_0,y_1)-d_X(x_1,y_0))},
\]
where $(x_0,y_0,x_1,y_1) \in X^4$ are points radially converging to the quadruple $(\xi_0,\eta_0,\xi_1,\eta_1)$. 

If $Y$ is another $\cat (-1)$ space, a topological embedding $\varphi: \bound X \rightarrow \bound Y$ between the boundaries at infinity is a \textit{Moebius map} if it satisfies 
\[
\bir_Y(\varphi(\xi_0),\varphi(\eta_0),\varphi(\xi_1),\varphi(\eta_1))=\bir_X(\xi_0,\eta_0,\xi_1,\eta_1)
\]
for every $\xi_0,\eta_0,\xi_1,\eta_1 \in \bound X$ distinct. Similary, a sequence of continuous maps $\varphi_k:\bound X \rightarrow \bound Y$ is \textit{asymptotically Moebius} if it holds
\[
\lim_{k \to \infty} \bir_Y(\varphi_k(\xi_0), \varphi_k(\eta_0), \varphi_k(\xi_1), \varphi_k(\eta_1))=\bir_X(\xi_0,\eta_0,\xi_1,\eta_1)
\]
for every $\xi_0,\eta_0,\xi_1,\eta_1 \in \bound X$ distinct. 
\end{deft}

\subsection{Visual metrics on the boundary at infinity}
The topology on $\bound X$ is metrizable and a family of natural distances can be constructed starting from the so-called Gromov product. Let $x,y,z$ be three points in $X$. The \textit{Gromov product} of $y$ and $z$ with respect to $x$ is given by $$(y|z)_x:=\frac{1}{2}(d_X(x,y)+d_X(x,z)-d_X(y,z)).$$ The product $(\cdot|\cdot)_x$ can be also extended to points at infinity as it follows. Let $\xi,\eta \in \bound X$ be two distinct points on the boundary of $X$ and let $p \in (\xi,\eta)$ be any point lying on the geodesic joining $\xi$ to $\eta$. Define $$(\xi|\eta)_x:=\frac{1}{2}(B_\xi(x,p)+B_\eta(x,p)),$$ where $B_\xi(x,p):=\lim_{t \to \infty} d_X(x,c(t))-d_X(p,c(t))$ is the Busemann function with respect to $\xi$ and $c$ is any geodesic ray representing the point $\xi$. It is easy to verify that if $y \in [x,\xi)$ and $z \in [x,\eta)$ then $$(\xi|\eta)_x= \lim_{(y,z) \to (\xi,\eta)} (y|z)_x.$$

This product at infinity allows us to define the \textit{visual metric pointed at} $x \in X$. If we fix $x \in X$ as an origin, we define $$\rho_x(\xi,\eta):=e^{-(\xi|\eta)_x}.$$ One can verify that the function $\rho_x$ define a metric on $\bound X$ which is bounded by $1$ and when the equality is attained the point $x$ must lie on the geodesic $(\xi,\eta)$. By changing the origin with a different point $y \in X$ we get another metric $\rho_y$ which is related to $\rho_x$ by the following formula
\[
\rho_y(\xi,\eta)=\rho_x(\xi,\eta)e^{\frac{1}{2}(B_\xi(x,y)+B_\eta(x,y))},
\]
for every $\xi,\eta \in \bound X$. The visual metric is intimately related with the cross ratio, indeed by~\cite[Proposition 1.1]{bourdon96:articolo} we have that 
\begin{equation}\label{visual}
\bir_X(\xi_0,\eta_0,\xi_1,\eta_1)=\frac{\rho_x(\xi_0,\xi_1)\rho_x(\eta_0,\eta_1)}{\rho_x(\xi_0,\eta_1)\rho_x(\eta_0,\xi_1)}
\end{equation}
for every $\xi_0,\eta_0,\xi_1,\eta_1 \in \bound X$ and for every $x \in X$. 

The visual metric $\rho_x(\cdot,\cdot)$ is also related with the comparison angle at infinity. Let $x,y,z$ be three distinct points in $X$ and let $\Delta:=[x,y] \cup [y,z] \cup [z,x]$ be the triangle associated to them. Let $\bar \Delta=[\bar x,\bar y] \cup [\bar y, \bar z] \cup [\bar z, \bar x]$ be a comparison triangle associated to $\Delta$ in $\hyp^2$. We define the \textit{comparison angle} at $x$ as
\[
\theta_x(y,z):=\angle_{\bar x}(\bar y,\bar z).
\]

As stated in~\cite[Proposition 1.2]{bourdon96:articolo} this definition can be extended to point of the boundary $\bound X$. More precisely let $\xi,\eta \in \bound X$ be two different points and let $y,z \in X$ radially converging to $\xi$ and $\eta$, respectively. Then $\theta_x(y,z)$ converges to a real number $\theta_x(\xi,\eta)$ called \textit{comparison angle at infinity}. Moreover it holds
\begin{equation} \label{angle}
\sin \frac{1}{2} \theta_x(\xi,\eta)=\rho_x(\xi,\eta),
\end{equation}
for every $\xi, \eta \in \bound X$ and every $x \in X$. 


\section{Proof of the main theorem}

In this section we are going to prove Theorem~\ref{convergence}. The proof will be essentially based on the strategy exposed by Bourdon in~\cite{bourdon96:articolo} for the hyperbolic plane. Here we are going to exploit the fact that the proof can be $\veps$-modified in order to hold also at infinity. 
Let $X$ be a $\cat (-1)$ space such that the isometry group $\isom(X)$ acts transitively on triples of distinct points of $\bound X$. Let $\varphi_k: \bound \hyp^2 \rightarrow \bound X$ be a sequence of continuous maps which is asymptotically Moebius. Since the isometry group acts transitively on triples, there exists a sequence of isometries $(g_k)_{k \in \bbN}$, $g_k \in \isom (X)$ such that
\begin{equation} \label{fixed0}
g_k\varphi_k(\xi_0)=\xi_0', \hspace{5pt} g_k\varphi_k(\eta_0)=\eta_0', \hspace{5pt} g_k\varphi_k(\xi_1)=\xi_1'
\end{equation}
for every $k \in \bbN$. Here $\xi_0',\eta_0',\xi_1' \in \bound X$ is a fixed triple of distinct points. If the triple $(\xi_0,\eta_0,\xi_1) \in (\bound \hyp^2)^3$ of distinct points satisfies Equation (\ref{fixed0}) we are going to say that it has fixed images. More generally, given a sequence of maps $\psi_k: \bound \hyp^2 \rightarrow \bound X$, we are going to say that a triple $(\xi_0,\eta_0,\xi_1) \in (\bound \hyp^2)^3$ of distinct points has \textit{fixed images} if it holds 
\begin{equation} \label{fixed}
\psi_k(\xi_0)=\xi_0', \hspace{5pt} \psi_k(\eta_0)=\eta_0', \hspace{5pt}\psi_k(\xi_1)=\xi_1'
\end{equation}
for every $k \in \bbN$ and for a fixed triple $\xi_0',\eta_0',\xi_1' \in \bound X$. 

\begin{lem}\label{existence} 
Let $\psi_k: \bound \hyp^2 \rightarrow \bound X$ be a sequence of asymptotically Moebius maps. Assume that $(\xi_0,\eta_0,\xi_1) \in \bound \hyp^2$ is a triple of distinct points with fixed images. Then for every other point $\eta_1 \in \bound \hyp^2$ the sequence $(\psi_k(\eta_1))_{k \in \bbN}$ admits a well-defined limit.  
\end{lem}

\begin{proof}
Assume the triple $(\xi_0,\eta_0,\xi_1)$ is positively oriented with respect to the counter clockwise orientation of the circle $\bound \hyp^2$. The circle is divided into three distinct intervals by the triple. More precisely, given two points $\xi,\eta \in \bound \hyp^2$, we define the interval $$I(\xi,\eta):=\{ \alpha \in \bound \hyp^2| \textup{the triple $(\xi,\alpha,\eta)$ is positively oriented}\}.$$ We are going to prove that the limit of the sequence $(\psi_k(\eta_1))_{k \in \bbN}$ exists for every $\eta_1 \in I(\xi_1,\xi_0)$. The proof for the points of the other two intervals $I(\xi_0,\eta_0)$ and $I(\eta_0,\xi_1)$ can be suitably adapted and hence we will omit it. 

Since we have chosen $\eta_1 \in I(\xi_1,\xi_0)$, it is clear that the geodesics $(\xi_0,\xi_1)$ and $(\eta_0,\eta_1)$ are incident. Let $s \in \hyp^2$ be the intersection point of these two geodesics and define $\alpha=\angle_s(\xi_0,\eta_0)$ the angle determined by $\xi_0$ and $\eta_0$. 

For every $k \in \bbN$ we can define a point $x_k$ on the geodesic $(\xi_0',\xi_1')$ which satisfies 
\[
\beta_k:=\theta_{x_k}(\xi_0',\eta_0')=\theta_{x_k}(\xi_1',\varphi_k(\eta_1)),
\]
where we recall that $\theta_{x_k}(\cdot,\cdot)$ is the comparison angle at infinity. 

Since the maps $\psi_k$ are asymptotically Moebius, by definition we have
\[
\lim_{k \to \infty} \bir_X(\xi_0',\psi_k(\eta_1),\eta_0',\xi_1')=\bir_{\hyp^2}(\xi_0,\eta_1,\eta_0,\xi_1)
\]
for every $\xi_0,\eta_0,\xi_1,\eta_1 \in \hyp^2$. Hence for a fixed $\veps>0$, there exists a $k_0 \in \bbN$ such that 
\[
|\bir_{\hyp^2}(\xi_0,\eta_1,\eta_0,\xi_1) - \bir_X(\xi_0',\psi_k(\eta_1),\eta_0',\xi_1')| < \veps
\]
 for every $k> k_0$. 

Thanks to Equation (\ref{visual}) and Equation (\ref{angle}) we have that 
\begin{align*}
\sin^2 \frac{\alpha}{2} &=\bir_{\hyp^2}(\xi_0,\eta_1,\eta_0,\xi_1) \geq \bir_X(\xi_0',\psi_k(\eta_1),\eta_0',\xi_1') - \veps = \\
                                       &=\frac{\rho_{x_k}(\xi_0',\eta_0')\rho_{x_k}(\xi_1',\psi_k(\eta_1))}{\rho_{x_k}(\xi_0',\xi_1')\rho_{x_k}(\eta_0',\psi_k(\eta_1))} - \veps \geq \\
			       & \geq \rho_{x_k}(\xi_0',\eta_0')\rho_{x_k}(\xi_1',\psi_k(\eta_1)) - \veps= \sin^2 \frac{\beta_k}{2} - \veps,
\end{align*}
assuming $k>k_0$. In the same way we have that 
\begin{align*}
\cos^2 \frac{\alpha}{2} &=\bir_{\hyp^2}(\xi_0,\eta_0,\eta_1, \xi_1) \geq \bir_X(\xi_0',\eta_0',\psi_k(\eta_1), \xi_1') - \veps = \\
                                       &=\frac{\rho_{x_k}(\xi_0',\psi_k(\eta_1))\rho_{x_k}(\xi_1',\eta_0'))}{\rho_{x_k}(\xi_0',\xi_1')\rho_{x_k}(\eta_0',\psi_k(\eta_1))} - \veps \geq \\
			       & \geq \rho_{x_k}(\xi_0',\psi_k(\eta_1))\rho_{x_k}(\xi_1',\eta_0') - \veps=\sin \frac{\theta_{x_k}(\xi_0',\psi_k(\eta_1))}{2} \sin \frac{\theta_{x_k}(\eta_0',\xi_1')}{2} - \veps,
\end{align*}
again assuming $k>k_0$.

Since the point $x_k$ lies on the geodesic $(\xi_0',\xi_1')$ for every $k \in \bbN$ and $X$ is a $\cat (-1)$ space, it must hold
\[
\begin{cases}
\theta_{x_k}(\xi_0',\psi_k(\eta_1)) \geq \pi - \beta_k,\\
\theta_{x_k}(\eta_0',\xi_1') \geq \pi-\beta_k,
\end{cases}
\]
for every $k \in \bbN$. As a consequence we get
\begin{align*}
\cos^2 \frac{\alpha}{2} &\geq \sin \frac{\theta_{x_k}(\xi_0',\psi_k(\eta_1))}{2} \sin \frac{\theta_{x_k}(\eta_0',\xi_1')}{2} - \veps\\
			        &\geq \sin \frac{\pi - \beta_k}{2} \sin \frac{\pi-\beta_k}{2} - \veps=\cos^2 \frac{\beta_k}{2}-\veps.
\end{align*}

To sum up, for a fixed value $\veps >0$, there exists $k_0 \in \bbN$ such that 
\[
\sin^2 \frac{\alpha}{2} - \veps \leq \sin^2 \frac{\beta_k}{2} \leq \sin^2 \frac{\alpha}{2} + \veps,
\]
which means that $$\lim_{k \to \infty} \theta_{x_k}(\xi_0,\eta_0)=\lim_{k \to \infty} \beta_k=\alpha.$$

Since for every $k \in \bbN$ we have chosen a point $x_k$ on the geodesic $(\xi_0',\xi_1')$  and there is a unique point on $x \in (\xi_0',\xi_1')$ such that $\theta_x(\xi_0,\eta_0)=\alpha$, by continuity of the comparison angle at infinity we must have
\[
\lim_{k \to \infty} x_k=x. 
\]

In the same way, there is a unique point $\eta_1' \in \bound X$ on the geodesic determined by $\eta_0'$ and $x$ such that $\theta_x(\xi_1',\eta_1')=\alpha$, thus it follows
\[
\lim_{k \to \infty} \psi_k(\eta_1)=\eta_1', 
\]
as desired. 
\end{proof}

We are now ready to prove the main theorem.

\begin{proof}[Proof of Theorem~\ref{convergence}]
Let $\varphi_k: \bound \hyp^2 \rightarrow \bound X$ be a sequence of asymptotically Moebius maps. Consider a triple $(\xi_0,\eta_0,\xi_1) \in \bound \hyp^2$ of distinct points. As we noticed at the beginning of the section, by hypothesis we can find a sequence of isometries $(g_k)_{k \in \bbN}$, $g_k \in \isom(X)$ such that $(\xi_0,\eta_0,\xi_1)$ has fixed images. By Lemma~\ref{existence} for every $\eta_1' \in \bound \hyp^2$ there exists $\eta_1' \in \bound X$ such that
\[
\lim_{k \to \infty} g_k \varphi_k(\eta_1)=\eta_1'.
\]

Define $\varphi_\infty:\bound \hyp^2 \rightarrow \bound X$ in this way, that means $\varphi_\infty(\eta_1):=\eta_1'$. We get a function $\varphi_\infty$ which is the limit of the sequence $(g_k\varphi_k)_{k \in \bbN}$. By continuity of cross ratio on $4$-tuples of distinct points this function is Moebius and the theorem follows from~\cite{bourdon96:articolo}. 
\end{proof} 

\begin{oss}
The existence of the sequence $(g_k)_{k \in \bbN}$ is necessary to guarantee the existence of the limit map $\varphi_\infty: \bound \hyp^2 \rightarrow \bound X$. If we require only that the maps are asymptotically Moebius, their limit might exist or not. For instance, let $\varphi: \bound \hyp^2 \rightarrow \bound \hyp^3$ be a Moebius map and let $(g_k)_{k \in \bbN}$ be a divergent sequence of loxodromic isometries $g_k \in \textup{PO}(3,1)^\circ$. It is clear that the map $\varphi_k:=g_k \circ \varphi$ are asymptotically Moebius since each one is actually Moebius, but the sequence has no limit. 
\end{oss}

Roughly speaking a sequence of asymptotically Moebius maps admits a Moebius limit function up to precomposing suitably with isometries of $X$. Thanks to the equivalence between Moebius property of a boundary map and the existence of a geodesic conjugacy stated for instance in~\cite{biswas15:articolo}, we get the following

\begin{cor}\label{conjugacy}
Let $\varphi_k: \bound \hyp^2 \rightarrow \bound X$ be a sequence of asymptotically Moebius maps. Then there exists a sequence $(g_k)_{k \in \bbN}$ of isometries of $X$ and a map $\varphi_\infty:\bound \hyp^2 \rightarrow \bound X$ such that $$\lim_{k \to \infty} g_k\varphi_k(\xi)=\varphi_\infty(\xi),$$ for every $\xi \in \bound \hyp^2$. Moreover $\varphi_\infty$ conjugates the geodesic flows on the respective unit tangent bundles. 
\end{cor}

It should be clear by what we have written so far that the procedure we used in the proof is specific for the hyperbolic plane. Even if we strongly believe that a statement similar to Theorem~\ref{convergence} should hold in the more general context of asymptotic Moebius maps of rank-one symmetric space of non-compact type into $\cat (-1)$ spaces, we would need something different from a simple generalization of the argument used in~\cite{bourdon96:articolo}. There the author shows that the geodesic conjugacy between the unit tangent bundle obtained by the Moebius map must be actually induced by a map between fiber bundles (that means it respect the points of tangent spaces) and to do this he exploits the validity of the result for the hyperbolic plane. Unfortunately in our context we cannot apply this reasoning because we can guarantee the convergence of only one isometrically embedded copy of $\hyp^2$, hence we cannot apply the transitive property used in the proof of Bourdon.


\addcontentsline{toc}{chapter}{\bibname} 	

\vspace{40pt}
Alessio Savini, Department of Mathematics, University of Bologna, Piazza di Porta San Donato 5, 40126 Bologna, Italy\\
\textit{E-mail address}: \texttt{alessio.savini5@unibo.it}
\end{document}